\documentclass[12pt]{amsart}

\usepackage{amssymb}
\usepackage{amsmath}
\numberwithin{equation}{section}
\usepackage{breakcites}
\usepackage{color}
\usepackage{graphicx}
\usepackage{epstopdf}
\usepackage{array}
\usepackage{hyperref}

\DeclareMathOperator{\sgn}{sgn}

\usepackage[vcentermath,enableskew]{youngtab}
\usepackage{ytableau}
\usepackage{amssymb}
\usepackage{rotating}
\usepackage{young}
\usepackage{enumerate}

\newcommand{\no}{\noindent}

\newcommand{\be}{\begin{equation}} 
\newcommand{\bea}{\begin{eqnarray}}
\newcommand{\ee}{\end{equation}}
\newcommand{\beas}{\begin{eqnarray*}}
\newcommand{\eea}{\end{eqnarray}}
\newcommand{\eeas}{\end{eqnarray*}}

\def\C{{\mathbb C}}

\def\z3{{\mathbb Z_3}}

\def\sg{\mathfrak S}
\newtheorem{theorem}{Theorem}[section]

\newtheorem{corollary}[theorem]{Corollary}

\newtheorem{lemma}[theorem]{Lemma}

\newtheorem{rem}[theorem]{Remark}
\begin{document}
\title{A new presentation for Specht modules with distinct parts}
\author[Friedmann]{Tamar Friedmann}
\address{Department of Mathematics, Colby College}
\email{tfriedma@colby.edu}

\author[Hanlon]{Phil Hanlon}
\address{Department of Mathematics, Dartmouth College}
\email{philip.j.hanlon@dartmouth.edu}

\author[Wachs]{Michelle L. Wachs$^*$}
\address{Department of Mathematics, University of Miami}
\email{wachs@math.miami.edu}
\thanks{$^{*}$Supported in part by NSF grant
 DMS 2207337.}

\begin{abstract}  
We obtain a new presentation for Specht modules whose conjugate shapes have strictly decreasing parts by introducing a linear operator on the space generated by column tabloids. The generators of the presentation are column tabloids and the relations form a proper subset of the Garnir relations of Fulton.  The results in this paper  generalize  earlier results of the authors and Stanley on Specht modules of staircase shape. 
\end{abstract}

\maketitle

\section{Introduction} \label{intro}

The Specht modules $S^\lambda$, where $\lambda$ is a partition of $n$, give a complete set of  irreducible 
representations of the symmetric group $\sg_n$ over a field  of characteristic $0$, say $\mathbb C $.  They can be constructed as  
subspaces of the regular representation $\mathbb C \sg_n$  or  as presentations given in terms of generators and relations, 
known as Garnir relations.   This paper deals primarily with the latter type of construction.

Let $\lambda= (\lambda_1 \geq\dots \geq \lambda_l)$ be a partition of $n$.  A {\em Young tableaux} of shape $\lambda$ is a filling of the Young diagram of shape $\lambda$ with distinct entries from the set $[n]:=\{1,2,\dots,n\}$.  Let $\mathcal T_\lambda$ be the set of  Young tableaux of shape
$\lambda$.

 To construct the Specht module as a submodule of the regular representation, one can use  Young symmetrizers.  For $t \in \mathcal T_\lambda$, the {\em Young symmetrizer} is defined by 
 \begin{equation} \label{symmeq} e_t:=\sum_{\alpha \in R_{t}}  \alpha  \\  \sum_{\beta \in C_t} \sgn(\beta) \beta ,\end{equation} 
where $C_t$ is the column stabilizer of $t$ and $R_t$ is the row stabilizer.
The {\em Specht module} $S^{\lambda}$ is the submodule of the regular representation  $\C \sg_{n}$ spanned by $\{\tau e_t : \tau \in \sg_{n} \} $.

To construct the Specht module as a presentation, one can use column tabloids and Garnir relations. Let $M^\lambda$ be the vector space (over $\C$) generated by  
$\mathcal T_\lambda$ subject only to column relations, which are of
the form $t+s$, where $s\in \mathcal T_\lambda$ is obtained from $t\in \mathcal T_\lambda$ by switching two
entries in the same column.  
Given $t \in \mathcal T_\lambda$, let $\bar t$ denote the coset of $t$
in $M^\lambda$. These cosets, which are called {\it column tabloids}, generate $M^{\lambda}$.
A Young tableau is {\em column strict} if the entries of each of its columns increase from top to bottom.  Clearly,
$\{\bar t  :  t \mbox{ is a column strict Young tableau of shape } \lambda\}$ is a basis for $M^\lambda$.

\vspace{.1in}\noindent {\bf Example:}
 $$ \ytableausetup
{mathmode, boxsize=1em}
\overline{\begin{ytableau}
\scriptstyle 3 & \scriptstyle 5  \\
 \scriptstyle 1 &  \scriptstyle 4  \\
 \scriptstyle 2 \\
\end{ytableau}}
\,\, = \,\, - \,
\ytableausetup
{mathmode, boxsize=1em}
\overline{\begin{ytableau}
\scriptstyle 1 & \scriptstyle 5  \\
 \scriptstyle 3 &  \scriptstyle 4  \\
 \scriptstyle 2 \\
\end{ytableau}}
\,\, = \,\, 
\ytableausetup
{mathmode, boxsize=1em}
\overline{\begin{ytableau}
\scriptstyle 1 & \scriptstyle 5  \\
 \scriptstyle 2 &  \scriptstyle 4  \\
 \scriptstyle 3 \\
\end{ytableau}}
\,\, = \,\, - \,
\ytableausetup
{mathmode, boxsize=1em}
\overline{\begin{ytableau}
\scriptstyle 1 & \scriptstyle 4  \\
 \scriptstyle 2 &  \scriptstyle 5  \\
 \scriptstyle 3 \\
\end{ytableau}}\,\,,$$
\vspace{.1in}

There are various different presentations of the Specht module $S^\lambda$ in the literature that involve  column tabloids and Garnir relations, see e.g. \cite{Fu, Sa}. Here we are interested in a presentation of $S^\lambda$ discussed in Fulton \cite[Section 7.4]{Fu}.  The generators are the column tabloids $\bar t$, where $t \in \mathcal T_\lambda$.
The Garnir relations are of the form
$\bar t-\sum \bar s$, where the sum is over all $s\in \mathcal
T_\lambda$ obtained from $t\in \mathcal T_\lambda$ by exchanging any
$k$ entries of a fixed  column  with the top $k$ entries of the next
column, while maintaining the vertical order of each of the exchanged
sets.  There is a Garnir relation $g^t_{c,k}$ for every $t \in
\mathcal T_\lambda$, every column $c \in [\lambda_1-1]$,  and  every
$k$ from $1$ to the length $l_{c+1}$ of  column $c+1$.  

\vspace{.1in}\noindent {\bf Example:} For
$$t= \ytableausetup
{mathmode, boxsize=1em}
\begin{ytableau}
\scriptstyle 1 & \scriptstyle 5 & \scriptstyle 7 \\
 \scriptstyle 2 &  \scriptstyle 6  \\
 \scriptstyle 3 \\
\scriptstyle 4
\end{ytableau}\,\,,$$

we have
$$ g^t_{1,1} =
\ytableausetup
{mathmode, boxsize=1em}
\overline{\begin{ytableau}
\scriptstyle 1 & \scriptstyle 5 & \scriptstyle 7 \\
 \scriptstyle 2 &  \scriptstyle 6  \\
 \scriptstyle 3 \\
\scriptstyle 4
\end{ytableau}}
\,\,- \,\,
\ytableausetup
{mathmode, boxsize=1em}
\overline{\begin{ytableau}
\scriptstyle 5 & \scriptstyle 1  & \scriptstyle 7 \\
 \scriptstyle 2 &  \scriptstyle 6  \\
 \scriptstyle 3 \\
\scriptstyle 4
\end{ytableau}}
\,\,- \,\,
\ytableausetup
{mathmode, boxsize=1em}
\overline{\begin{ytableau}
\scriptstyle 1 & \scriptstyle 2  & \scriptstyle 7\\
 \scriptstyle 5 &  \scriptstyle 6  \\
 \scriptstyle 3 \\
\scriptstyle 4
\end{ytableau}}
\,\,- \,\,
\ytableausetup
{mathmode, boxsize=1em}
\overline{\begin{ytableau}
\scriptstyle 1 & \scriptstyle 3 & \scriptstyle 7 \\
 \scriptstyle 2 &  \scriptstyle 6  \\
 \scriptstyle 5 \\
\scriptstyle 4
\end{ytableau}}
\,\,- \,\,
\ytableausetup
{mathmode, boxsize=1em}
\overline{\begin{ytableau}
\scriptstyle 1 & \scriptstyle 4 & \scriptstyle 7 \\
 \scriptstyle 2 &  \scriptstyle 6  \\
 \scriptstyle 3 \\
\scriptstyle 5
\end{ytableau}}
$$

and
$$ g^t_{1,2} =
\ytableausetup
{mathmode, boxsize=1em}
\overline{\begin{ytableau}
\scriptstyle 1 & \scriptstyle 5 & \scriptstyle 7 \\
 \scriptstyle 2 &  \scriptstyle 6  \\
 \scriptstyle 3 \\
\scriptstyle 4
\end{ytableau}}
\,\,- \,\,
\ytableausetup
{mathmode, boxsize=1em}
\overline{\begin{ytableau}
\scriptstyle 5 & \scriptstyle 1 & \scriptstyle 7 \\
 \scriptstyle 6 &  \scriptstyle 2  \\
 \scriptstyle 3 \\
\scriptstyle 4
\end{ytableau}}
\,\,- \,\,
\ytableausetup
{mathmode, boxsize=1em}
\overline{\begin{ytableau}
\scriptstyle 5 & \scriptstyle 1 & \scriptstyle 7\\
 \scriptstyle 2 &  \scriptstyle 3  \\
 \scriptstyle 6 \\
\scriptstyle 4
\end{ytableau}}
\,\,- \,\,
\ytableausetup
{mathmode, boxsize=1em}
\overline{\begin{ytableau}
\scriptstyle 5 & \scriptstyle 1 & \scriptstyle 7\\
 \scriptstyle 2 &  \scriptstyle 4  \\
 \scriptstyle 3 \\
\scriptstyle 6
\end{ytableau}}
\,\,- \,\,
\ytableausetup
{mathmode, boxsize=1em}
\overline{\begin{ytableau}
\scriptstyle 1 & \scriptstyle 2  & \scriptstyle 7\\
 \scriptstyle 5 &  \scriptstyle 3  \\
 \scriptstyle 6 \\
\scriptstyle 4
\end{ytableau}}
\,\,- \,\,
\ytableausetup
{mathmode, boxsize=1em}
\overline{\begin{ytableau}
\scriptstyle 1 & \scriptstyle 2  & \scriptstyle 7\\
 \scriptstyle 5 &  \scriptstyle 4  \\
 \scriptstyle 3 \\
\scriptstyle 6
\end{ytableau}}
\,\,- \,\,
\ytableausetup
{mathmode, boxsize=1em}
\overline{\begin{ytableau}
\scriptstyle 1 & \scriptstyle 3  & \scriptstyle 7\\
 \scriptstyle 2 &  \scriptstyle 4  \\
 \scriptstyle 5 \\
\scriptstyle 6
\end{ytableau}}\, \, .
$$

\vspace{.1in} Let $G^{\lambda}$ be the subspace of $M^\lambda$ generated by the Garnir relations in
 \begin{equation} \label{garnireq} \{ g^t_{c,k} : c \in [\lambda_1-1], k \in [l_{c+1}], t \in \mathcal T_\lambda \}  .\end{equation}  
The symmetric group $\sg_n$ acts on $\mathcal T_\lambda$ by replacing
 each entry of a tableau by its image under the permutation in $\sg_n$.
 This induces a representation of $\sg_n$ on $M^\lambda$. Clearly $G^\lambda$ is invariant under the action of
$\sg_n$. 
The  presentation of $S^\lambda$ obtained in Section 7.4 of \cite{Fu} is given by
   \be \label{Fultoneq} M^\lambda/G^\lambda \cong_{\sg_n} S^\lambda. \ee

 On page 102 (before Ex.~16) of \cite{Fu}, a presentation of $S^\lambda$ with a smaller set of relations is given.  In this presentation, the index $k$ in $g^t_{c,k}$ of (\ref{garnireq}) is restricted to a single value:  $k=\min [l_{c+1}]= 1$. More precisely, the presentation is
\be \label{1fulton} M^{\lambda} / G^{\lambda,\min} \cong_{\sg_n} S^{\lambda} ,\ee where
 $G^{\lambda,\min}$ is the subspace of $G^\lambda$ generated by the subset of Garnir relations, $$\{ g^t_{c,1} : c \in [\lambda_1-1], t \in \mathcal T_\lambda \}  .$$ 
 
 In this paper we obtain an analogous presentation of $S^\lambda$, also with a smaller set of relations, when the conjugate $\lambda^*$ has distinct parts. The index $k$ in $g^t_{c,k}$ of (\ref{garnireq}) is again restricted to a single value, but now that value is the maximum value:  $k=\max [l_{c+1}]= l_{c+1}$.
The presentation 
is given in our main result:
\begin{theorem} \label{mainth} Let $\lambda$ be a partition whose conjugate has distinct parts. Then as $\sg_n$-modules,
\begin{equation} \label{maineq} M^{\lambda} / G^{\lambda,\max} \cong S^{\lambda} ,\end{equation} where
 $G^{\lambda,\max}$ is the subspace of $G^\lambda$ generated by $$\{ g^t_{c,l_{c+1}} : c \in [\lambda_1-1], t \in \mathcal T_\lambda \}  .$$
Moreover, we can further reduce this set of relations by restricting $t$ to the set of column strict tableaux.
\end{theorem} 

Theorem~\ref{mainth}  is obtained as a consequence of Corollary~\ref{eigencor1} below, which gives all eigenspaces of a certain linear operator on $M^{\lambda}$, where $\lambda$ is any two-column shape $ 2^a1^{b}$.  This eigenspace result improves   an earlier result of the authors and Stanley \cite{FHSW2} (see also \cite{FHSW1}), in the special case that $b=1$. The earlier result  was presented in the setting of an $n$-ary generalization of Lie algebra called a LAnKe or 
Filippov algebra, where only the $b=1$ case was relevant.  An observation in   \cite{FHSW2} that the $n$-ary Jacobi relations correspond to the restricted class of Garnir relations $G^{\lambda,\max}$  for  $\lambda=2^{n-1}1^{1}$  provided a special case of Theorem \ref{mainth} for staircase shapes and motivated the  work in the current paper. Our proof in this paper of the eigenspace result corrects an error in the proof for the special case given in \cite{FHSW2}.

The method of introducing a linear operator on the space of column tabloids was subsequently used in  \cite{BF}  to obtain a different presentation of $S^\lambda$ with a reduced number of relations. This presentation works for {\it all} 
shapes, but rather than using a subset of the Garnir relations,  it uses a set consisting of sums of  the Garnir relations that 
generate $G^{\lambda, \min}$.

In the next section we prove  our eigenspace result and  use it to prove Theorem~\ref{mainth}.  Actually, we obtain a slightly more general version of  Theorem~\ref{mainth}  along with a converse.

\section{The proof of Theorem~\ref{mainth}: reduction to the two-column case}

Theorem~\ref{mainth} holds more generally when we allow the conjugate $\lambda^*$ to have multiple parts equal to $1$, while still requiring the parts greater than $1$ to be distinct.   Note that a partition $\lambda=(\lambda_1\ge \lambda_2 \ge  \dots\ge  \lambda_l)$ meets this requirement if  and only if $\lambda_i -\lambda_{i+1} \le 1$ for all $i=2,\dots l-1$ and $\lambda_{l} =1$. For example,  $\lambda = (5,4,2,1,1,1)^*= (6,3,2,2,1)$ meets this requirement, but  $\lambda = (4,3,2,2)^*= (4,4,2,1)$ does not.  The following result implies Theorem~\ref{mainth}.  

\begin{theorem} \label{mainth2} Let $\lambda = (\lambda_1 \ge \lambda_2 \ge \cdots \ge \lambda_\ell)$ be a partition of $n$.  If  $\lambda_i -\lambda_{i+1} \le 1$ for each $i = 2,\dots,l-1$ and $\lambda_l=1$ then  as $\sg_n$-modules,
\begin{equation} \label{maineq} M^{\lambda} / G^{\lambda,\max} \cong S^{\lambda} ,\end{equation} where
 $G^{\lambda,\max}$ is the subspace of $G^\lambda$ generated by $$\{ g^t_{c,l_{c+1}} : c \in [\lambda_1-1], t \in \mathcal T_\lambda \}  .$$
Moreover, we can further reduce this set of relations by restricting $t$ to the set of column strict tableaux.  Conversely, if (\ref{maineq}) holds then $\lambda_i -\lambda_{i+1} \le 1$ for each $i = 2,\dots,l-1$ and $\lambda_l=1$.
\end{theorem} 

It  follows from Fulton's presentation given in (\ref{Fultoneq}) that we need only  prove Theorem~\ref{mainth2} for shapes  with   just two columns.  Indeed, let $\lambda$ be a partition of $n$. By (\ref{Fultoneq}) $M^\lambda / G^{\lambda, \max} = S^\lambda$ if and only if $G^{\lambda, \max} = G^\lambda$.  
Since $G^{\lambda, \max}$ equals the direct sum over all columns $ c$ of the subspaces spanned by  $\{g^t_{c,l_{c+1}}:  t \in \mathcal T_\lambda\}$, and $G^{\lambda}$ equals the direct sum over all columns $ c$ of the subspaces spanned by  the larger set $\{g^t_{c,k}: k \in [l_{c+1}], t \in \mathcal T_\lambda\}$, we have that $G^{\lambda, \max} = G^\lambda$ if and only if 
\begin{equation} \label{twocoleq} \langle g^t_{c,l_{c+1}}:  t \in \mathcal T_\lambda\rangle= \langle g^t_{c,k}: k \in [l_{c+1}], t \in \mathcal T_\lambda\rangle\end{equation} for all columns $c$.
By applying (\ref{Fultoneq})  to the  conjugate of $(l_c,l_{c+1})$ we have that (\ref{twocoleq}) holds if and only if 
$$M^{(l_c,l_{c+1})^*} / G^{(l_c,l_{c+1})^*,\max} = M^{(l_c,l_{c+1})^*} / \langle g^t_{c,l_{c+1}}:  t \in \mathcal T_{(l_c, l_{c+1} )^*}\rangle = S^{(l_c,l_{c+1})^*}.$$
Putting all this together yields $M^\lambda / G^{\lambda, \max} = S^\lambda$ if and only if $$M^{(l_c,l_{c+1})^*} / G^{(l_c,l_{c+1})^*,\max} = S^{(l_c,l_{c+1})^*}$$ for all columns $c$ of $\lambda$.  Hence our claim that we  need only consider  two-column shapes holds.

For the remainder of this section we consider only two-column shapes $2^m1^{n-m}$.
For $n \ge m$, let $V_{n,m} := M^{2^m1^{n-m}}$. For each $n$-element subset $S$ of $[n+m]$,  let $t_S$ be  the column strict Young tableau of shape $2^m1^{n-m}$ whose first column consists of the elements of $S$ and   second column consists of the elements of $[n+m] \setminus S $, and write  $v_S$ for the column  tabloid  $ \bar t_S$ indexed by $S$.  
Clearly, 
\begin{equation} \label{basiseq} \left \{ v_S : S \in \binom{[n+m]}{n}\right \}\end{equation} 
is a basis for $V_{n,m}$.  Thus $V_{n,m}$ has dimension $\binom{n+m}{n}$.

For each $S \in \binom{[n+m]}{n}$,  let $g_S$ denote the Garnir relation $g^{t_S}_{1,m}$, that is
$$g_S :=  v_S \, -\sum \bar t ,$$ where the sum is over all tableaux
$t$ obtained from $t_S$ by  choosing an $m$-subset $S^\prime$ of $S$ and for each $i=1,2,\dots,m$, exchanging the $i$th smallest  element of $S^\prime$  (which is in column 1 of $t_S$) with the $i$th element of column 2 of $t_S$.   
 For example, if $n=3$ and $m=2$  then   

\begin{align*}g_{\{2,4,5\}} &=
\ytableausetup
{mathmode, boxsize=1em}
\overline{\begin{ytableau}
\scriptstyle 2 & \scriptstyle 1  \\
 \scriptstyle 4 &  \scriptstyle 3  \\
 \scriptstyle 5
\end{ytableau}}
\,\,- \,\,
\overline{\begin{ytableau}
\scriptstyle 1 & \scriptstyle 2  \\
 \scriptstyle 3 &  \scriptstyle 4  \\
 \scriptstyle 5
\end{ytableau}}
\,\, - \,\,
\overline{\begin{ytableau}
\scriptstyle 1 & \scriptstyle 2  \\
 \scriptstyle 4 &  \scriptstyle 5  \\
 \scriptstyle 3
\end{ytableau}}
\,\, - \,\,
\overline{\begin{ytableau}
\scriptstyle 2 & \scriptstyle 4  \\
 \scriptstyle 1 &  \scriptstyle 5  \\
 \scriptstyle 3
\end{ytableau}} \,\,.
\\
&= v_{\{2,4,5\} }-v_{\{1,3,5\}} +v_{\{1,3,4\}} +v_{\{1,2,3\}}.
\end{align*}

\begin{lemma} \label{coefth} For all $v \in V_{n,m}$ and $S \in \binom{[n+m]}n$, let $\langle v,v_S \rangle$ denote the coefficient of $v_S$ in the expansion of $v$ in the basis given in (\ref{basiseq}).  Also let $p:= n-m$. Then for all $S,T \in \binom{[n+m]}{n}$,
$$\langle g_S,v_T \rangle = \left\{
  \begin{array}{l l}

1& \quad  \text{if } S=T  \\
\\

    (-1)^{d_1+ \dots + d_{p} +\binom{p +1}2+1}&  \quad  \text{if } S\cap T = \{d_1,\dots, d_{p}\}\\
\\ 
    0 & \quad \text{if } S\neq T \text{ but } |S\cap T| >p . \
  \end{array} \right.$$
\end{lemma}
\begin{proof} It is easy to see that  the only possible cases are: $S=T$, $|S\cap T| = p$  and $|S\cap T|  > p$. 
 It is straightforward that the  result for the first and third cases holds.  For the second case, let $S = \{a_1,a_2,\dots,a_n \}$ 
 and $[n+m] \setminus S = \{b_1,b_2,\dots,b_m \}$, where the $a$'s and $b$'s are listed in increasing order. 
 Then the terms of $g_S$  other than $v_S$ are of the form $-\bar t$, where $t$ is a tableau whose   first column listed from top 
 to bottom is  $$b_1, \dots, b_{i_1-1}, a_{i_1}, b_{i_1}, \dots, b_{i_2-2}, a_{i_2}, b_{i_2-1},\dots,
 b_{i_p-p}, a_{i_p}, b_{i_p-p+1}, \dots ,b_m $$ with  $1\le i_1<i_2<\dots<i_p\le n$, and whose second column is $[n+m] \setminus T$ listed in increasing order,
  where $T=\{  a_{i_1}, \dots, a_{i_p}, b_1, \dots, b_m\}$. 
 Clearly $\bar t$ is equal to $(-1)^{i_1+\dots+i_p-\binom {p+1}2}$ times the tabloid   whose first column from top to bottom is 
 $$ a_{i_1}, \dots, a_{i_p}, b_1, \dots, b_m$$ and whose second column is $[n+m] \setminus T$ listed in increasing order.  This tabloid is equal to 
 $(-1)^{j_1+j_2+\dots+j_p} v_T$, where  each  $j_k$ is the number of $b$'s  that 
 are less than $a_{i_k}$.   Since 
 $i_k + j_k= a_{i_k}$ for each $k$,  we have
 $$\langle g_S, v_T\rangle = (-1)^{a_{i_1} + a_{i_2} + \dots +a_{i_p}-\binom{p+1}2 +1}.$$ Now note that  $\{a_{i_1}, a_{i_2}, \dots, a_{i_p}\} = S \cap T$.
\end{proof}

Let $G_{n,m}$ be the subspace of $V_{n,m}$ generated by $\{g_S : S \in \binom{[n+m]} n\}$.  It is not difficult to see that $G_{n,m}$ is invariant under the action of $\sg_{n+m}$ on $V_{n,m}$. The two-column case of  Theorem~\ref{mainth2} can now be stated as follows.
\begin{theorem} \label{twocolth} Let $m \le n$. Then  as $\sg_{n+m}$-modules,
$$V_{n,m}/G_{n,m} \cong S^{2^m1^{n-m}}$$ if and only if $m <n$ or $m=n=1$.
\end{theorem} 
We will  introduce a linear operator that will enable us to prove this result (and thereby prove  Theorems~\ref{mainth2} and~\ref{mainth}).  First we note that
due to the column relations, as  an $\sg_{n+m}$--module, $V_{n,m}$ is isomorphic to the representation of $\sg_{n+m}$ induced from the sign representation of the Young subgroup $\sg_n \times \sg_m$:
\[ V_{n,m} \cong_{\sg_{n+m}}  (\sgn_n \times \sgn_{m} ) \uparrow _{\sg_n \times \sg_{m}}^{\sg_{n+m}}.  \]  
Hence by  Pieri's rule,
\begin{equation} \label{freeeq} V_{n,m} \cong_{\sg_{n+m}} \bigoplus _{i=0}^{m} S^{2^i1^{n+m-2i}}.\end{equation}

  Now consider the linear operator $\varphi: V_{n,m} \to V_{n,m}$ defined on  basis elements by
$$\varphi(v_S) = g_S.$$  It is not difficult to see that $\varphi$ is an $\sg_{n+m}$-module homomorphism whose image is $G_{n,m}$.

 \begin{lemma} \label{decomplem} Let $m \le n$.
The operator $\varphi$ acts as multiplication by a scalar on each irreducible submodule of $V_{n,m}$.
Moreover, as  $\sg_{n+m}$--modules, \begin{equation} \label{kereq} \ker \varphi \cong V_{n,m} / G_{n,m}. \end{equation}
\end{lemma}
\begin{proof} 
This follows from the fact that $V_{n,m}$ is multiplicity-free, as indicated in  (\ref{freeeq}), and Schur's Lemma.
\end{proof}

It follows from (\ref{kereq}) that to prove Theorem~\ref{twocolth}, we need only show that the kernel of $\varphi$ is isomorphic to $S^{2^m1^{n-m}}$ if and only if $m <n$  or $m=n=1$.  This is handled by the following theorem and its corollaries. 

\begin{theorem} \label{eigens} Let $m \le n$.
The operator $\varphi:V_{n,m} \to V_{n,m}$ acts as multiplication by the scalar 
\be \label{dneqn} w _i := 1-{n-i \choose m-i}(-1)^{m-i},\ee
on the irreducible submodule  isomorphic to $S^{2^i1^{(n+m)-2i}}$ for $i=0,1,\dots,m$.
\end{theorem}

 \begin{rem} In the case that $m=n-1$, Theorem~\ref{eigens} reduces to Theorem~2.4 of \cite{FHSW2}, which was given in the setting  of $n$-ary Jacobi relations, where only the $m=n-1$ case arises.     
\end{rem}

\begin{corollary} \label{eigencor1}
For $m<n$, the operator $\varphi$ has $m+1$ distinct eigenvalues $w_0, w_1,\dots,w_m$.
Moreover, if $E_i$ is the eigenspace corresponding to $w _i$ then as $\sg_{n+m}$-modules,
$$E_i\cong S^{2^i1^{(n+m)-2i}}$$
for each $i=0,1, \ldots , m$. Consequently, 
\begin{equation} \label{nulleq} \ker \varphi \,\, \cong_{\sg_{n+m}}\,\, S^{2^m1^{n-m}}.
\end{equation}
\end{corollary}

\begin{corollary} \label{eigencor2}
 For $m=n$ the operator $\varphi$ has the eigenvalues $0$ and $2$, with
$$ w_i= 
\begin{cases} 
0 & \text{if } n-i \text{ is even}, \\
2& \text{if } n-i \text{ is odd},
\end{cases}
$$
for each $i=0, \ldots , n$. Consequently, 
\be \label{d=n} \ker \varphi \,\, \cong_{\sg_{n+m}}  \,\, \bigoplus^n_
{\begin{subarray}{c} i=0 \\ n-i \text{ even} \end{subarray}} S^{2^i1^{2n-2i}},
\ee
and the eigenspace with eigenvalue 2 is given by
\[  \bigoplus^n_
{\begin{subarray}{c} i=0 \\ n-i \text{ odd} \end{subarray}} S^{2^i1^{2n-2i}}.
\]
\end{corollary}

\begin{proof}[Proof of Theorem~\ref{eigens}] 
The proof   follows (and at the same time corrects) the argument for the  $m=n-1$ case given in \cite[Theorem 2.4]{FHSW2}.  Indeed, if we set $m=n-1$ in the following proof, we get  a corrected version of the proof in \cite{FHSW2}, which was given in the language of  $n$-ary Jacobi relations, where only the $m=n-1$ case was relevant.   
On occasion, we  will refer back to the arguments in the proof of \cite[Theorem 2.4]{FHSW2}.

By Lemma~\ref{decomplem}, $\varphi$ acts as a scalar on each irreducible submodule. 
 To compute the scalar, we start by letting $t$ be the standard Young tableau of shape $2^i1^{n+m-2i}$ whose first column is the concatenation of the sequences $1,2,\dots,n$ and $n+i+1,n+i+2, \dots, n+m$ and whose second column is the sequence $n+1, n+2, \dots, n+i$.  
Now set $r_t:= \sum_{\alpha \in R_t} \alpha$ and $c_t:= \sum_{\beta \in C_t}\sgn(\beta) \beta$, where  $R_t$ and $C_t$ denote the row and column stabilizer of $t$, respectively.  Recall from (\ref{symmeq}) that  $e_t = r_t c_t$  and that the Specht module $S^{2^i 1^{n+m-2i}}$ is the submodule of the regular representation  $\C \sg_{n+m}$ spanned by $\{\tau e_t : \tau \in \sg_{n+m} \} $.

Now let $F_t$ be the set of permutations $\sigma$ in $C_t$ that  fix the elements of  $\{ n+1, n+2, \dots , n+i\}$ and satisfy $$\sigma(1)<\cdots <\sigma(n), \hskip .3cm  \sigma(n+i+1)<\cdots <\sigma(n+m),$$  and let $$f_t:=\sum_{\sigma \in F_t} \sgn(\sigma) \sigma. $$
Also let $T :=\{1,2, \dots, n\}$.  We 
claim\footnote{In the proof of \cite[Theorem 2.4]{FHSW2}, the same claim was made {\em erroneously}   for $f_t:=\sum_{\sigma \in F_t} \sgn(\sigma) \sigma^{-1}$ in the case $m=n-1$. Here we provide a proof of the correct claim and then proceed with an argument analogous to what is given in   \cite{FHSW2}. } that  
\be \label{cfeq} c_t v_T= n!(m-i)! i! f_t v_T.\ee
Indeed, $$c_t = \sum_{\sigma \in F_t} \sgn(\sigma) \sigma  \sum_{\tau \in \sg_A \times \sg_B \times \sg_C } \sgn(\tau) \tau,$$
where  $A=\{1,\dots,n\}$, $B= \{n+i+1,\dots, n+m\}$, and $C=\{n+1,\dots,n+i\}$. It follows from the antisymmetry of the columns of $\tau v_T$ that 
$$c_t v_T = n!(m-i)! i!\sum_{\sigma \in F_t} \sgn(\sigma) \sigma v_T = n!(m-i)! i! f_t v_T,$$
as claimed.

It follows from (\ref{cfeq}) that   $r_t f_t v_T$  is a  scalar multiple of $e_t v_T$. 
Since the coefficient of $v_T$ in the expansion of $r_tf_t v_T$ is $1$, we have $e_tv_T \ne 0$.
Now following the proof of \cite[Theorem 2.4]{FHSW2}, we can show  that the subspace spanned by $\{\pi e_t v_T : \pi \in \sg_{n+m} \} $ is the unique submodule of $V_{n,m}$ isomorphic to $S^{2^i1^{(n+m)-2i}}$. 
This allows us to abuse notation by letting $S^{2^i1^{(n+m)-2i}}$ denote this submodule of $V_{n,m}$.

Since $r_t f_t v_T$ is a scalar multiple of $e_tv_T$, it is in $S^ {2^i1^{(n+m)-2i}}$.  It follows that 
$$\varphi(r_t f_t v_T ) = c r_t f_t v_T ,$$ for some scalar $c$, which we want to show equals $w_i$.
 Similarly to the proof of \cite[Theorem 2.4]{FHSW2}, we can  use Lemma~\ref{coefth} to show that 
$$c=\langle \varphi(r_t f_t v_T), v_T \rangle = 1+  \sum_{\substack{S \in \binom{[n+m]}{n} \setminus \{T\}\\ S\cap T =p}} \langle r_t f_t v_T , v_S\rangle \langle\varphi(v_S),v_T\rangle.$$
Let $S \cap T=  \{d_1,d_2,\dots,d_{p}\}$,  in which case $ \langle\varphi(v_S),v_T\rangle = (-1)^{d_1+ \dots + d_{p} +\binom {p+1}2+1}$.  Hence 
\be \label{dsumeq} c = 1+ \sum_{D \in \binom{[n]}{ p} }(-1)^{\left (\sum_{d \in D} d \right ) +\binom{p+1}2+1} \langle r_t f_t v_T , v_{S(D)}\rangle,\ee
where  $$S(D) = D \cup \{n+1,n+2,\dots,n+m\}.$$

To compute $ \langle r_t f_t v_T , v_{S(D)}\rangle$, we must consider how we get $v_{S(D)}$ from the action on $v_T$ of a permutation  $\alpha \sigma$, where $\alpha \in R_t$ and $\sigma \in F_t$.
In order to get $S(D)$ for some $D:=\{d_1<d_2<\dots<d_p\}  \in \binom{[n]}{ p} $, we must have that $D \in \binom{\{i+1,\dots,n\}} p$ and that $\sigma$ fixes the elements of $\{n+1,\dots, n+i\}$ and  interchanges 
$\{n+i+1, \ldots , n+m\}$ with $\{i+1,\dots,n\} \setminus D$ putting $D$ in positions $i+1,\dots,i+p$. More precisely, in one line notation, $\sigma$ restricted to the first column of $t$ is the concatenation of  the sequences $$ 1,\dots, i $$ $$ d_1,\dots,d_p$$ 
$$ n+i+1,\dots, n+m$$
$$i+1, \ldots , d_1-1, d_1+1, \ldots, d_2-1,d_2+1, \ldots, d_p-1, d_p+1, \ldots, n.$$
In one line notation, $\sigma$ restricted to the second column is  $n+1, \dots, n+i $.
 We must also have
$$\alpha = (1, n+1)(2, n+2)\cdots (i, n+i),$$
giving us $S(D)$ in the first $n$ entries of the first column of $\alpha\sigma t$.

The number of inversions in $\sigma$ restricted to the first column of $t$ is $$\sum_{k=1}^p (d_k-k -i )+ (m-i)(n-i-p) = \sum_{k=1}^p d_k -\binom {p+1}2 -ip+ (m-i)^2,
$$  and $\sigma$ restricted to the second column of $t$ has no inversions.
It follows that 
\be \label{sgneq} \sgn(\sigma) = (-1)^{\left (\sum_{k=1}^p d_k \right)-\binom {p+1}2 -ip + (m-i)}.\ee

In terms of our basis, the first column of 
$ \alpha \sigma v_T$ (written horizontally) is   $$n+1, \ldots , n+i, {\color{red} d_1}, \ldots , {\color{red} d_p}, n+i+1,  \ldots , n+m$$ and the second column is $$1, \ldots , i, i+1, \ldots, d_1-1, d_1+1, \ldots, d_p-1, d_p+1, \ldots ,n . $$
To put this basis in canonical form, we need to put the first column in  increasing order by moving all the $d_k$'s all the way to the top of the column. That requires $i$ transpositions for each $d_k$, resulting in the sign $(-1)^{ip}$. The second column is already in increasing order.
This yields $ \alpha \sigma v_T =  (-1)^{ip} v_{S(d_1,d_2,\dots,d_p)}$.  
Hence by (\ref{sgneq}), $$\sgn(\sigma) \alpha \sigma v_T = (-1)^{\left (\sum_{k=1}^p d_k \right)-\binom {p+1}2  + (m-i)}v_{S(d_1,d_2,\dots,d_p)}.$$
We can now conclude that for all $D\in \binom{[n]}p$,
$$\langle r_t f_t v_T, v_{S(D)} \rangle = \begin{cases}  (-1)^{\left (\sum_{d \in D} d \right) -{p+1\choose 2} +(m-i)} &\mbox{ if } D \in \binom{\{i+1,\dots ,n\}} p \\
0 &\mbox{ otherwise.}\end{cases} $$
By substituting this into (\ref{dsumeq}), we get the eigenvalue,
$$c = 1- \sum_{D \in \binom{\{i+1,\dots ,n\}} p} (-1)^{m-i} = 1 - \binom{n-i}{n-m} (-1)^{m-i} = w_i, $$
which is all that is needed to complete the proof of the theorem.
\end{proof}

 \begin{proof}[Proof of Theorem~\ref{twocolth}]  
We can now use  (\ref{kereq}), (\ref{nulleq}), and the $n=1$ case of (\ref{d=n}) to conclude that the ``if" direction of Theorem~\ref{twocolth} holds.  The ``only if" direction also follows from  (\ref{kereq}), (\ref{nulleq}), and   (\ref{d=n}). 
\end{proof}

\vspace{.1in}\noindent {\it Final Remarks.} Corollary~\ref{eigencor2}  has another interesting consequence.  Indeed, the decomposition  of   $\ker\varphi$ given in (\ref{d=n}) yields a  decomposition of $V_{n,n}/G_{n,n}$ into irreducibles.  Since $V_{n,n}/G_{n,n}$ is clearly isomorphic to the composition product  of the trivial representation of $\sg_2$ and the sign representation  of $\sg_n$, we recover a well-known result on the decomposition of this composition product into irreducibles; see e.g. \cite[p. 140]{Mac}.

One may ask what happens if the  $k$ in the Garnir relations  (\ref{garnireq}) is restricted to a single value other than the minimum 1 and the maximum $l_{c+1}$.   For which partitions $\lambda$ do such Garnir relations generate the entire set of  Garnir relations? Is there a
generalization of Theorem~\ref{mainth2} that answers this question?  We leave these questions open.\footnote{Motivated by an earlier version of our paper,
Maliakas, Metzaki, and  Stergiopoulou address these questions in~\cite{MMS}.  These questions are also addressed in a  forthcoming paper of Friedmann.}

\vskip .5cm
\no \large {\bf Acknowledgements}
\normalsize
\\
We are grateful to Mihalis Maliakas for pointing out an error in the statement of Theorem~\ref{mainth2} in an earlier version of this paper.  We also thank the anonymous referees for their helpful comments.

\end{document}